\documentclass{amsart}
\pdfoutput=1
\usepackage{amsfonts,amscd,amssymb,amsmath,amsthm}
\usepackage{booktabs}
\usepackage{graphicx,caption,subcaption,mathrsfs,appendix}
\usepackage{xparse,centernot}
\usepackage[T1]{fontenc}

\usepackage[foot]{amsaddr}

\begin{document}

\newtheorem{theorem}{Theorem}[section]
\newtheorem{lemma}[theorem]{Lemma}
\newtheorem{corollary}[theorem]{Corollary}
\newtheorem{conjecture}[theorem]{Conjecture}
\newtheorem{cor}[theorem]{Corollary}
\newtheorem{proposition}[theorem]{Proposition}
\theoremstyle{definition}
\newtheorem{definition}[theorem]{Definition}
\newtheorem{example}[theorem]{Example}
\newtheorem{claim}[theorem]{Claim}
\newtheorem{remark}[theorem]{Remark}

\newenvironment{pfofthm}[1]
{\par\vskip2\parsep\noindent{\sc Proof of\ #1. }}{{\hfill
$\Box$}
\par\vskip2\parsep}
\newenvironment{pfoflem}[1]
{\par\vskip2\parsep\noindent{\sc Proof of Lemma\ #1. }}{{\hfill
$\Box$}
\par\vskip2\parsep}

%%%%%%%%%%%%%%%%%%%%%%%%%%%%%%%%%%%%%%%%%%
%%% General macros
%%%%%%%%%%%%%%%%%%%%%%%%%%%%%%%%%%%%%%%%%%

\newcommand{\R}{\mathbb{R}}
\newcommand{\T}{\mathcal{T}}
\newcommand{\C}{\mathcal{C}}
\newcommand{\G}{\mathcal{G}}
\newcommand{\Z}{\mathbb{Z}}
\newcommand{\Q}{\mathbb{Q}}
\newcommand{\E}{\mathbb E}
\newcommand{\N}{\mathbb N}

\newcommand{\barray}{\begin{eqnarray*}}
\newcommand{\earray}{\end{eqnarray*}}

\newcommand{\beq}{\begin{equation}}
\newcommand{\eeq}{\end{equation}}

%%%%%%%%%%%%%%%%%%%%%%%%%%%%%%%%%%%%%%%%%%
%%% Probability Macros
%%%%%%%%%%%%%%%%%%%%%%%%%%%%%%%%%%%%%%%%%%

\renewcommand{\Pr}{\mathbb{P}}
\newcommand{\Prob}{\Pr}
\newcommand{\Exp}{\mathbb{E}}
\newcommand{\expect}{\mathbb{E}}
\newcommand{\1}{\mathbf{1}}
\newcommand{\prob}{\Pr}
\newcommand{\pr}{\Pr}
\newcommand{\filt}{\mathscr{F}}
\DeclareDocumentCommand \one { o }
{%
\IfNoValueTF {#1}
{\mathbf{1}  }
{\mathbf{1}\left\{ {#1} \right\} }%
}
\newcommand{\Bernoulli}{\operatorname{Bernoulli}}
\newcommand{\Binomial}{\operatorname{Binom}}
\newcommand{\Binom}{\Binomial}
\newcommand{\Poisson}{\operatorname{Poisson}}
\newcommand{\Exponential}{\operatorname{Exp}}

%%%%%%%%%%%%%%%%%%%%%%%%%%%%%%%%%%%%%%%%%%
%%% Random Graph/Complex Macros
%%%%%%%%%%%%%%%%%%%%%%%%%%%%%%%%%%%%%%%%%%

\newcommand{\link}{\mbox{lk}}
\newcommand{\Deg}{\operatorname{deg}}
\newcommand{\vertexsetof}[1]{V\left({#1}\right)}
\renewcommand{\deg}{\Deg}
\newcommand{\oneE}[2]{\mathbf{1}_{#1 \leftrightarrow #2}}
\newcommand{\ebetween}[2]{{#1} \leftrightarrow {#2}}
\newcommand{\noebetween}[2]{{#1} \centernot{\leftrightarrow} {#2}}
\newcommand{\Gap}{\ensuremath{\tilde \lambda_2 \vee |\tilde \lambda_n|}}
\newcommand{\dset}[2]{\ensuremath{ e({#1},{#2})}} 
\newcommand{\EL}{{ L}}
\newcommand{\ER}{{Erd\H{o}s--R\'{e}nyi }}
\newcommand{\zuk}{{\.{Z}uk}}

%%%%%%%%%%%%%%%%%%%%%%%%%%%%%%%%%%%%%%%%%%
%%% Paper-Specific Macros
%%%%%%%%%%%%%%%%%%%%%%%%%%%%%%%%%%%%%%%%%%

\newcommand{\frm}{\ensuremath{ 2\log\log m}}
\DeclareDocumentCommand \fuzz { o o }
{
\IfNoValueTF {#1}
  { \aleph_M }
  { \IfNoValueTF { #2 }
    { \aleph_{M}^{{#1}} }
    { \aleph_{M}^{{#1}}({#2})}
  } 
}
\newcommand{\csubzero}{c_{0}}
\newcommand{\csubone}{c_{1}}
\newcommand{\csubtwo}{c_{2}}
\newcommand{\csubthree}{c_{3}}
\newcommand{\csubstar}{c_{*}}
\newcommand{\rsp}{1-C\exp(-md^{1/4}\log n)}
\newcommand{\lc}{\ensuremath{ \operatorname{light}(x,y)}}
\newcommand{\hc}{\ensuremath{ \operatorname{heavy}(x,y)}}
\DeclareDocumentCommand \pam { O{m} }
{
P_{{#1}}
}

\title{The Power of 2 Choices over Preferential Attachment}
\author{Yury Malyshkin}
\address{Department of Mathematics and Mechanics, Moscow State University\\
Laboratory of Solid State Electronics, Tver State University}
\email{yury.malyshkin@mail.ru}
\author{Elliot Paquette}
\address{Department of Mathematics, Weizmann Institute of Science}
\email{elliot.paquette@gmail.com}
\thanks{YM gratefully acknowledges the support of the Weizmann Institute of Science, where this work was performed.
EP gratefully acknowledges the support of NSF Postdoctoral Fellowship DMS-1304057.
}
\date{\today}
\maketitle

\begin{abstract} 
We introduce a new type of preferential attachment tree that includes choices in its evolution, like with Achlioptas processes.  At each step in the growth of the graph, a new vertex is introduced.  Two possible neighbor vertices are selected independently and with probability proportional to degree.  Between the two, the vertex with smaller degree is chosen, and a new edge is created.  We determine with high probability the largest degree of this graph up to some additive error term.
\end{abstract}

\section{Introduction}
In the present work we consider an alteration of the preferential attachment model, in the spirit of the Achlioptas processes (see~\cite{Achlioptas,Riordan}).  The preferential attachment graph is a time-indexed sequence of graphs constructed the following way.  We start with a single edge, and at each time step we add a new vertex.  We then select an old vertex with probability proportional to the degree of the vertex, and we add a new edge between the new vertex and the selected vertex.
 This model is widely studied and many of its properties are known, such as the maximum degree, the limiting degree distribution, and the diameter of the graph (for instance see \cite{barabasi,FFF04,DvdHH,Mori}). In particular, in \cite{FFF04} it was shown that at time $t$, for any function $f$ with $f(t)\rightarrow\infty$ as $t\rightarrow\infty$, $\frac{t^{1/2}}{f(t)}\leq\Delta(t)\leq t^{1/2}f(t)$ with high probability, where $\Delta(t)$ is the highest degree of the preferential attachment graph at time $t$. In \cite{Mori}, this was strengthened to say that over the course of all time, $\Delta(t)t^{-1/2}$ converges almost surely to a non-degenerate positive random variable.  We say that some event $\mathcal{E}_{n}$ occurs with high probability as $n\rightarrow\infty$ if $\Pr(\mathcal{E}_{n})\rightarrow 1$ as $n\rightarrow\infty$. When it is clear which parameter is turning to infinity we omit it.

We will consider an alteration of this model that allows limited choice into its evolution.  Let us define a sequence of trees $\{ \pam \}$ given by the following rule.  Let $P_1$ be the one-edge tree.  Given $P_{m-1},$ define $P_m$ by first adding one new vertex $v_{m+1}$.  Let $X^1_m$ and $X^2_m$ be i.i.d. vertices from $\vertexsetof{ \pam }$ (here $\vertexsetof{P}$ is the set of vertices of $P$) chosen with probability 
\[
\Pr \left[
X^1 = w 
\right] = \frac{\deg w}{2m}.
\]
Note that as the graph has $m$ edges, $\sum_{w}\deg w=2m$. Finally, create a new edge between $v_{m+1}$ and $Y_m,$ where $Y_m$ is whichever of $X^1_m$ and $X^2_m$ has smaller degree. In the case of a tie, choose according to an independent fair coin toss.  We call this the \emph{min-choice preferential attachment tree}.  

%A two choices modification of a model is a strong tool for reducing maximal degree of the graph (see \cite{MRS01} for details). 

In \cite{DSKrM}, similar models of randomly evolving networks were introduced.  Among others, they study a model in which one again chooses two vertices $X^1_m$ and $X^2_m$ and chooses the minimal degree vertex.  However, they study the case where these vertices are picked with uniform probability.  

This is in turn strongly related to the original model of \cite{ABKU99}, in which this type of choice was introduced to study load balancing. In its simplest form, this amounts to studying balls thrown randomly into bins.  Suppose we have $n$ bins and $n$ balls, and on each step we put a new ball into one of the bins, choosing the bin randomly and uniformly.  In this model the number of balls in the most loaded bin is about $\log n/\log \log n,$ as $n\rightarrow\infty$. Adding two choices to this model significantly reduces this number. More precisely, we alter the model so that at each step we independently select two bins and put the ball in the bin that contains fewer balls. In the case that they hold the same number of balls, we choose the bin according to an independent fair coin toss. As a result the number of balls in the most loaded bin is $\log \log n/\log 2 + \Theta (1)$.  

\begin{figure}
\centering
	\begin{subfigure}[t]{0.45\textwidth}
	\includegraphics[width=\textwidth]{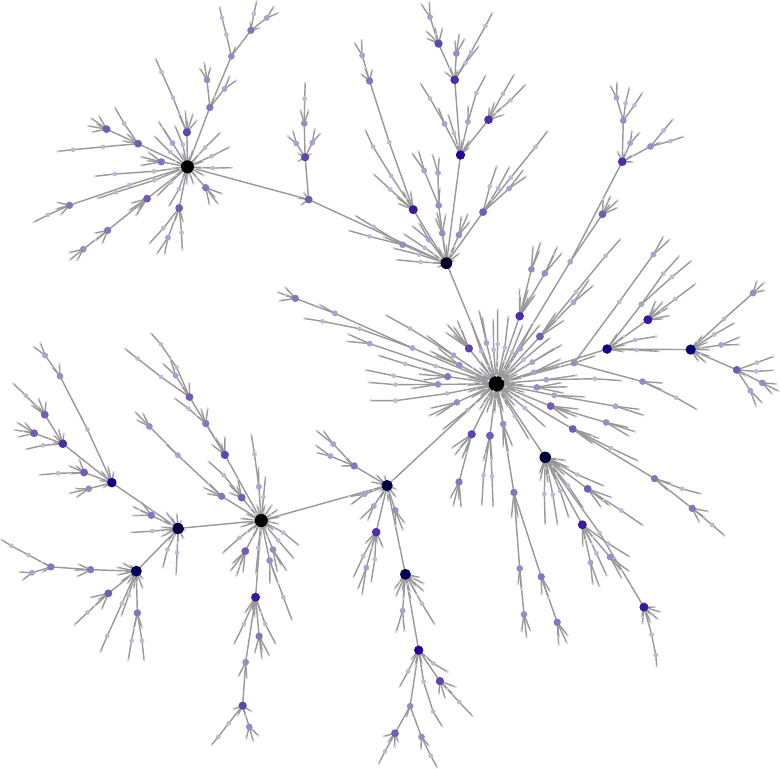}
	\caption{The preferential attachment tree after $1000$ vertices have been added.}
	\end{subfigure}
	\begin{subfigure}[t]{0.45\textwidth}
	\includegraphics[width=\textwidth]{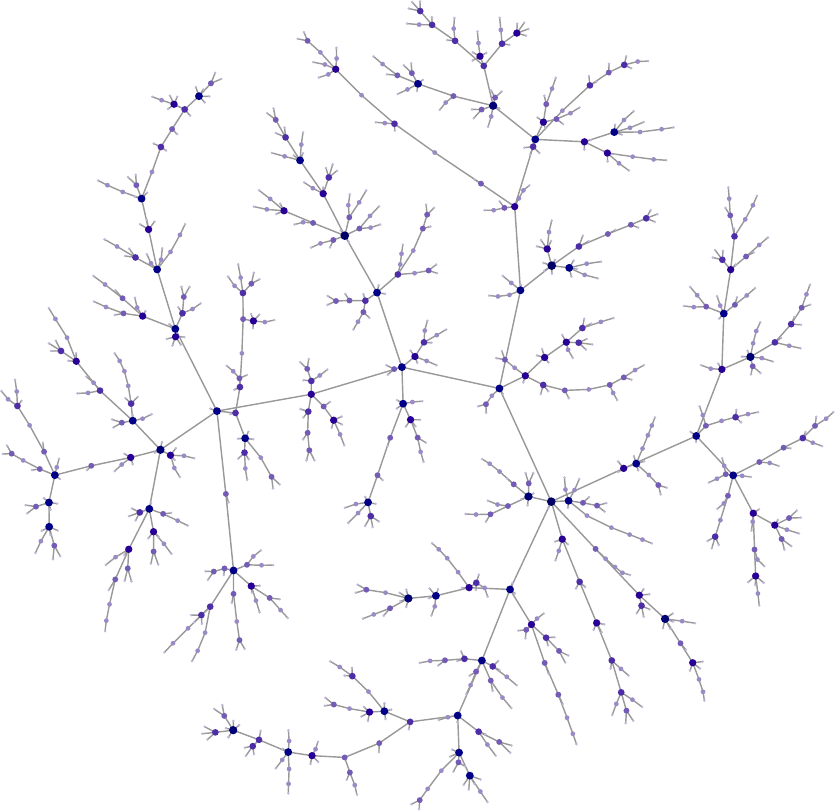}
	\caption{The min-choice preferential attachment tree after $1000$ vertices have been added.}
	\end{subfigure}
\end{figure}

There are a few differences between our model and the bin and ball model with two choices. First, the two-choice preferential attachment model tends to select higher degree vertices because of the size biasing.  Second, the ball and bin model tends to select empty bins frequently at the beginning of the process, while adding a new vertex to the two-choice preferential attachment model always increases the degree of an existing vertex (this is also true in the model of~\cite{DSKrM}, but it alone does not greatly increase the maximum degree).  Both influences tend to create higher degree vertices and more loaded bins.  Note that the combined influences of these effects have a large impact in the models without two choices.  The degree distribution in the preferential attachment model follows a power law~\cite{barabasi}, while the load distribution in the bin and ball model can be checked to have exponential tails.

Our main theorem shows that these differences are in some sense less powerful than the power of two choices.
\begin{theorem}
\label{thm:max_degree}
With high probability, the maximum degree of $P_m$ is $\frac{\log \log m}{\log 2} + \Theta(1).$
\end{theorem}

Before going deep into the proof, we will outline the approach.  Define $F_m(k)$ to be the weight under the size bias distribution given to vertices of the graph $P_{m}$ of degree greater than $k,$ i.e.
\[
F_m(k) = \sum_{i=1}^{m}  (\deg v_i) \one[\deg v_i \geq k].
\]
Note that $F_m(1) = \sum_{i=1}^m \deg v_i = 2m,$ as there are always $m$ edges in the graph.  If it holds that $F_{m}(k)>0$ for some $k>0,$ there is a vertex of $P_{m}$ with degree at least $k,$ while if $F_{m}(k)<k$ then all vertices of $P_{m}$ have degrees less then $k$. We will get an estimate on the maximal degree by controlling $F_{j}(k)/2j$. 

Now $F_m(k)$ as a function of $k$ is a Markov chain in $m$ which evolves according to the following rule, valid for $k > 1,$
\begin{equation}
\label{eq:evolution}
F_{m+1}(k) - F_m(k) =
\begin{cases}
1, & \Pr = \left( \frac{F_m(k)}{2m}\right)^2 \\
k, & \Pr = 
\left( \frac{F_m(k-1)}{2m}\right)^2
-\left( \frac{F_m(k)}{2m}\right)^2
\\
0, & \text{otherwise}. 
\end{cases}
\end{equation}
The key structure we use is that good control over $F_j(k-1)$ for some range of $j$ yields better control over $F_j(k)$ after waiting long enough for averaging to take effect.

For small $k,$ we need some initial estimate.  Thus for $k>1$ we define the function $\rho(k,t)$ given by
\begin{equation}
\label{eq:a_kdef}
\rho(k,t) =
\frac{\sqrt{4+(k-1)k\cdot t^2} - 2}{k-1}.
\end{equation}
We let $\alpha_1 = 2$ and define $\alpha_k = \rho(k,\alpha_{k-1}),\;k\geq 2$ inductively. 

These $\alpha_k$ decay doubly exponentially, but only after a long enough burn-in time.  For these initial steps, very careful analysis is required to ensure that they even decrease.  For this reason, we begin by making estimates for the first ten $\alpha_k.$   

\noindent The following rational upper bounds are easily verified inductively using~\eqref{eq:a_kdef} and monotonicity.
\begin{figure}[h]
\caption{Rational upper bounds for $\alpha_k,$ $1 \leq k \leq 10.$}
\label{fig:initial}
\begin{tabular}{l*{10}{c}}
Exact Value &
$\alpha_1$ &
$\alpha_2$ &
$\alpha_3$ &
$\alpha_4$ &
$\alpha_5$ &
$\alpha_6$ &
$\alpha_7$ &
$\alpha_8$ &
$\alpha_9$ &
$\alpha_{10}$ \\
\midrule
Bound &
$2$ &
$\tfrac32$ &
$\tfrac98$ &
$\tfrac45$ &
$\tfrac35$ &
$\tfrac25$ &
$\tfrac14$ &
$\tfrac18$ &
$\tfrac{1}{30}$ &
$\tfrac{1}{300}$ 
\end{tabular}
\end{figure}

 From random walk comparisons, we can show that $F_j(k)/j$ is nearly $\alpha_k$ holding $k$ fixed and making $j$ large.  As a corollary, we get the convergence of the empirical degree distribution of the tree (see Remark~\ref{rem:edd}).  After gaining some initial control, we continue by improving the estimates for larger and larger $k.$  In all, we go through 4 steps.
\begin{enumerate}
    \item We get starting estimates for $k$ less than some fixed $k_{0}>0$ and $\omega(m)\leq j\leq m,$ $\omega(m)\rightarrow\infty$ as $m\rightarrow\infty$ (see Lemma~\ref{lem:proc_start}).
    \item We get improved estimates for $k_{0}\leq k\leq k_{*}(m)$ that 
decrease doubly exponentially in $k$ but are only valid for
$\phi(m,k)\leq j\leq m,$ where $\phi(m,k)$ increases extremely rapidly in $k$ (see Lemma~\ref{lem:diagonal_recurrence}).
    \item We then get estimates of the form $F_j(k) \leq 2j^{1-\beta}$ for some $0<\beta<1$ that hold for $k > k_{*}(m)$ and $(\log\log m)^{M}\leq j\leq m.$  By increasing $k$ finitely many times, we can make $\beta$ very close to $1$ (see Lemma~\ref{lem:supercharger}).
    \item Once $\beta$ is sufficiently large for some $k = k_*(m) + r$, we show that in fact $F_j(k+1)$ must be $0$ (see Lemma~\ref{lem:final_recurrence}).
\end{enumerate}  

\section{Discussion}

Theorem~\ref{thm:max_degree} answers a question about the degree sequence of the tree, which uses no topological features of the graph.  In the case of the standard preferential attachment model, the diameter is known to be logarithmic~\cite{Pittel, DvdHH}.  It would be interesting to know if this remains the case in the min-choice preferential attachment tree or if the diameter is larger.  In~\cite{RuToVa}, the authors derive the limiting law of the preferential attachment tree viewed from a random vertex; a deeper, narrower tree should be expected in the case of the min-choice tree.

The \emph{max-choice} preferential attachment model also presents an interesting model.  This corresponds to choosing the vertex of larger degree instead of smaller degree.  For this model, we conjecture the largest degree of the tree with $m$ edges is of order $m / \log m.$  It would also be interesting to see if the two choices had a significant impact on the diameter of the graph.
%This model should also show phenomena such as the ``persistent hub,'' a single vertex that in some finite random time becomes the largest degree vertex for all time after (see~\cite{DeMo, Galashin}).  With such a hub, it is possible the diameter behaves in an even more extreme way.  

The preferential attachment model fits naturally inside a larger class of processes where the new vertex chooses a neighbor in the old graph with probability proportional to some power $\alpha$ of the degree, which was first studied in~\cite{KrReLe}.  In the case that $\alpha > 1,$ the tree has a single dominant vertex~\cite{OlSp}.  This ``persistent hub'' (using terminology of \cite{DeMo}) has degree of order $m,$ while all other vertices have bounded degree.  The min-choice adaptation can be made to these models as well, first sampling two vertices with probability proportional to the power $\alpha$ of the degree and then choosing the vertex with minimal degree.  Simulations suggest that for $\alpha $ large enough (around $1.8$) a single vertex dominates the others, while for $\alpha$ up to $1.5$ the tree remains more diffuse.  This leaves open the possibility of a sharp transition in behavior for some critical value of $\alpha.$

Note that the proof remains the same if instead of two random choices we consider $d$ random choices, where $d\geq 2$ is a fixed natural number. In this case, with high probability the maximum degree will be $\frac{\log \log m}{\log d} + \Theta(1).$  One interesting question is whether or not we obtain bounded maximum degree if we increase $d$ over the course of the process.  We conjecture that if $d$ is of order $\log m$ ( precisely, $d=\lfloor A\log m\rfloor$, where $A$ some positive constant) that the max-degrees of the min-choice preferential attachment trees $P_m$ are tight.
%that for any $\epsilon > 0$ there is a large enough constant $C(\epsilon)>0$, so the maximum degree does not exceed this constant with probability at least $1-\epsilon$. 
It is not clear if this is true for all $A>0$, or if there is some critical $A_{0}$, starting with which the maximum degree has this property. 

\section{Proofs}
For the first step we prove the following.
\begin{lemma}
\label{lem:proc_start}
For any $\epsilon > 0,$ any $\omega(m) \to \infty,$ $\omega < m$ and any $k \geq 1$ fixed, we have that 
\[
\Pr \left[
\exists j,~m \geq j \geq \omega(m)~:~ | F_j(k) - \alpha_k j | > \epsilon j 
\right] \to 0.
\]
\end{lemma}
\begin{remark}
\label{rem:edd}
Using the initial estimates in Figure~\ref{fig:initial} together with the bound that $\alpha_k < k \alpha_{k-1}^2,$ we can easily establish that $\alpha_k \to 0$ as $k \to \infty.$  Hence, from this lemma we get the tightness of the empirical degree distribution and its weak convergence to the distribution described by $\alpha_k$ as $j\to \infty.$
\end{remark}
\begin{proof}

We prove this lemma using induction over $k$.
The base case, \(k = 1\), is immediate as \(F_{j}(1) = 2j\) for all \(j\).
Define the event \(\mathcal{A}\) by
\[
\mathcal{A}(\omega_{0}(m), k-1,\delta)=
\{~\forall~j, \omega_{0}(m) \leq j< m, |F_{j}(k-1)-ja_{k-1}|<\delta j\}.
\]
From the induction hypothesis, we have that \(\mathcal{A}\) holds with high probability for any $\delta > 0$ fixed and any $\omega_0(m) \to \infty.$  Therefore, it suffices to show that there is a $\delta >0$ and a $\omega_0(m) \to \infty$ so that
\[
\Pr\left[
\mathcal{A}(\omega(m),k,\epsilon)^c \cap
\mathcal{A}(\omega_{0}(m), k-1,\delta)
\right] \to 0.
\]
We will only prove the upper bound, i.e. that \(F_j(k) - \alpha_k j \leq \epsilon j\) with high probability.  The lower bound follows from an identical argument.

Let $\omega_0(m) \to \infty$ and $\delta>0$ be considered fixed, with appropriate values to be determined later.
For $j$ such that $\omega_{0}(m)\leq j\leq m,$
\[
F_{j}(k)=
F_{\omega_{0}(m)}+
\sum_{i=\omega_{0}}^{j} \leq 2\omega_{0}(m)+
\sum_{i=\omega_{0}}^{j}\chi_{i},
\]
almost surely,
where 
\[
\chi_{i}=F_{i}(k) - F_{i-1}(k) =
\begin{cases}
1, & \Pr = \left( \frac{F_{i-1}(k)}{2(i-1)}\right)^2 \\
k, & \Pr = 
\left( \frac{F_{i-1}(k-1)}{2(i-1)}\right)^2
-\left( \frac{F_{i-1}(k)}{2(i-1)}\right)^2
\\
0, & \text{otherwise}. 
\end{cases}
\]
Note that $\frac{F_{i-1}(k-1)}{2(i-1)}\leq a_{k-1}/2+\delta/2$. 
We may contruct variables $\eta_{i}$ whose law given \( \sigma(F_{i-1}(k)) \)
is 
\[
\eta_{i}=
\begin{cases}
1, & \Pr = \left( \frac{F_{i-1}(k)}{2(i-1)}\right)^2 \\
k, & \Pr = 
\left( a_{k-1}/2+\delta/2\right)^2
-\left( \frac{F_{i-1}(k)}{2(i-1)}\right)^2
\\
0, & \text{otherwise}
\end{cases}
\]
so that on the event
\(
\mathcal{A}(\omega_{0}(m), k-1,\delta),
\)
we have $\chi_{i}\leq\eta_{i}$.
 
Then it follows that
\[
F_{j}(k)\leq 2\omega_{0}(m)+\sum_{i=\omega_{0}}^{j}\eta_{i}.
\]

Let $\pi$ be the first $j \geq \omega_{0}(m)$ so that $F_{j}(k) \leq (a_{k}+\epsilon/2)j$. We will estimate the probability that $\pi\leq\omega(m)$. Set $g_{i}=F_{i}(k)/(2i)-a_{k}$. If $\omega_{0}(m)\leq i <\pi$, then $g_{i}>\epsilon/2.$ 

We can expand the law of $\eta_{i}$ as
\[
\eta_{i}=
\begin{cases}
1, & \Pr = a_{k}^{2}/4+g_{i}a_{k}/2+g_{i}^{2}/4 \\
k, & \Pr = a_{k-1}^{2}/4 - a_{k}^{2}/4 + \delta a_{k-1}/2 + \delta^{2}/4 - g_{i}a_{k}/2 - g_{i}^{2}/4
\\
0, & \text{otherwise}. 
\end{cases}
\]
Choose $\delta$ such that $\delta a_{k-1}/2 + \delta^{2}/4=\epsilon/(4k)$. 
We may construct i.i.d. variables 
\[
\eta_{i}^{'}=
\begin{cases}
1, & \Pr = a_{k}^{2}/4\\
k, & \Pr = a_{k-1}^{2}/4 - a_{k}^{2}/4+\epsilon/(4k)
\\
0, & \text{otherwise} 
\end{cases}
\]
so that $\eta_i \leq \eta_i'$ on 
\(\mathcal{A}(\omega_{0}(m), k-1,\delta)\)
for $i>\omega_{0}(m)$ such that $F_{i}(k)/i\geq a_{k}.$  Set $\rho$ to be the first time after $\pi$ that $F_i(k)/i \leq a_k.$

Note that from the definition of $a_{k},$ it follows that $\mathbb{E}\eta_{i}^{'}=a_{k}+\epsilon/4$. Now we obtain the estimate
\begin{align*}
\mathbb{P}(\pi>\omega(m)) 
&\leq \mathbb{P}(\sum_{i=\omega_{0}}^{\omega(m)} \eta_{i}+2\omega_{0}(m)>a_{k}\omega(m)+\epsilon\omega(m)/2) \\
&\leq \mathbb{P}(\sum_{i=\omega_{0}}^{\omega(m)} \eta_{i}^{'}>a_{k}(\omega(m)-\omega_{0}(m))+\epsilon\omega(m)/2-\omega_{0}(m)(2-a_{k})).
\end{align*}
Choose 
\(
\omega_{0}(m)=\frac{3\epsilon}{8(2-a_{k}-\epsilon/8)}\omega(m),
\) so that
\(
\epsilon\omega(m)/2-\omega_{0}(m)(2-a_{k})=(\omega(m)-\omega_{0}(m))\epsilon/8.
\)
Then we have
\[
\mathbb{P}(\pi>\omega(m)) \leq \mathbb{P}(\sum_{i=\omega_{0}}^{\omega(m)} \eta_{i}^{'}>(a_{k}+\epsilon/8)(\omega(m)-\omega_{0}(m)))\leq C_{1}e^{-C_{2}\omega(m)},
\]
where $C_{1}, C_{2}$ are some positive constants (which still depend on $k$ and $\epsilon$).

Now we estimate the probability that $F_{j}(k)$ reaches the line $a_{k}j+\epsilon j$ when started from time $\pi>\omega_{0}$.  From monotonicity, we may assume that \(F_\pi(k) =\lfloor a_k \pi + \epsilon \pi/2\rfloor.\) Let $M_a(j)$ denote the random walk with increments distributed as $\eta_1'$, started from level $a$nand stopped when the process crosses the line $a_kj.$ 

Define the following function
\[
p(m,r_1,r_2) = \sup_{t \geq \omega_0(m)}
\Pr\left[
\exists~j \geq t
~:~
M_{\lfloor a_k t + r_1 t\rfloor}(j-t) \geq \lfloor(a_k + r_2) j \rfloor
\right].
\]
We claim that for all fixed \( \epsilon/4 < r_1 < r_2 \), we have \( p(m,r_1,r_2) \to 0\).  This follows from a simple tail bound estimate, and we will delay the proof until the end. 

Let $\rho_1$ be the first time after $\pi$ that the process drops below the line $a_kj$ and returns to level greater than $\lfloor a_{k}j+\epsilon j/2 \rfloor-k$ without crossing $a_{k}j+\epsilon j.$  
Likewise, let $\rho_{i} \geq \rho_{i-1}$ be the $i^{th}$ time that this happens.  Given that $\rho_{i} < \infty,$ for $\rho_{i}$ to occur, it must be that the process crosses from level 
\(
\lfloor a_{k}j+\epsilon j/3 \rfloor 
\)
to level
\(
\lfloor a_{k}j+3\epsilon j/8 \rfloor, 
\)
provided $m$ is sufficiently large,
and hence 
\begin{equation*}
\Pr \left[
\rho_i < \infty ~|~ \rho_{i-1} < \infty
\right] \leq p(m,\epsilon/3,3\epsilon/8).
\end{equation*}
We now decompose the probability of $F_j(k)$ exceeding \(a_k j + \epsilon j\) according to the renewal times \(\rho_j\).
\begin{align*}
\Pr\left[
\exists~j \geq \pi~:~
F_j(k) > a_k j + \epsilon j
\right]
\hspace{-1.5in}& \\
&\leq 
\sum_{i=0}^\infty 
\Pr\left[
\exists~j, \rho_{i+1} \geq j \geq \rho_i~:~
F_j(k) > a_k j + \epsilon j 
~
\middle\vert
~
\rho_i < \infty
\right]
\Pr\left[
\rho_i < \infty
\right] \\
&\leq 
\sum_{i=0}^\infty 
p(m,\epsilon/2,\epsilon)
p(m,\epsilon/3,3\epsilon/8)^i
=o(1).
\end{align*}

It remains to show that for all fixed \( \epsilon/4 < r_1 < r_2 \), we have \( p(m,r_1,r_2) \to 0\).
Set $S_j = M_a(j) - \Exp M_a(j).$  The event that 
\[
\mathcal{E}
=
\{
\exists~j \geq t~:~
M_{\lfloor a_k t + r_1 t\rfloor}(j-t) \geq \lfloor(a_k + r_2) j \rfloor
\}
\]
has
\[
\mathcal{E}
\subseteq
\{
\exists~n \geq 0~:~
S_{n} \geq (r_2 - \epsilon/4)n + (r_2-r_1)t - 1
\}.
\]
From Hoeffding's inequality, we have that for fixed $n,$ there is a constant $c=c(k,\epsilon) > 0$ so that
\[
\Pr \left[
S_n \geq t
\right] \leq \exp(-c t^2/n).
\]
Summing over all $n,$ we get that 
\begin{align*}
\Pr \left[
\exists~n \geq 0~:~
S_{n} \geq (r_2 - \epsilon/4)n + (r_2-r_1)t - 1
\right]
\hspace{-2in}& \\ 
&\leq \sum_{n=0}^\infty
\exp(-c( (r_2 - \epsilon/4)n + (r_2 - r_1)t)^2/n)  \\
&\leq 
\exp(-2c(r_2 - \epsilon/4)(r_2 - r_1)t) 
\sum_{n=1}^\infty
\exp(-c(r_2 - \epsilon/4)n) \\
&\leq 
\frac{\exp(-2c(r_2 - \epsilon/4)(r_2 - r_1)t)}
{1-\exp(-c(r_2 - \epsilon/4)}.
\end{align*}
This goes to $0$ uniformly in $t \geq w_0(m),$ and hence the proof is complete.

\end{proof}
\noindent Now, let $k_{0}=10.$ Let $f(k_0) = \frac{1}{100}$ and inductively define
\(
f(k+1)=f(k)^2(k+1)
\)
for $k \geq k_0.$ 
\begin{lemma}
\label{lem:f_estimate}
There are constants $c_1 > 0$ and $c_2 > 0$ so that 
for all $j \geq 0,$
\[
\exp( -c_1 2^j ) \leq 
f(k_0 + j) \leq \exp( -c_2 2^j ).
\]
\end{lemma}
\begin{proof}

It is easily verified by induction that $f(k)$ can be expressed using the following rule for $k > k_0,$
\begin{equation}
\label{eq:f_exact}
\log f(k) = 2^{k-k_0}\sum_{i=1}^{k-k_0} 2^{-i} \log(k_0 + i) + 2^{k-k_0}\log f(k_0).
\end{equation}
Thus, from the positivity of the $\log(k_0 + i)$ term, it follows immediately that 
\[
\log f(k) > 2^{k-k_0}\log f(k_0),
\]
so that the lower bound holds with $c_1 = -\log f(k_0).$ For the upper bound, we note that $\log(k_0 + i) \leq \log(k_0) + i$ and hence
\[
\sum_{i=1}^{k-k_0} 2^{-i} \log(k_0 + i)
\leq
\sum_{i=1}^{\infty} 2^{-i} (\log(k_0) + i)
=\log(k_0) + 2.
\]
Thus from~\eqref{eq:f_exact}, we have that 
\[
\log f(k) \leq 2^{k-k_0}\left( \log(k_0) + 2 + \log(f(k_0)) \right),
\]
As we have $e^2k_0f(k_0) < 1,$ we may take $c_2 = -\log (e^2k_0f(k_0))$ to complete the proof.

\end{proof}

Now set $\rho(m) = \lceil (\log \log m)^{1/3} \rceil$ and define $\phi(m,k)$ to be $\rho(m) C^{2^{k+1}}$ where $C$ is an integer sufficiently large that
\begin{equation}
\label{eq:Cdef}
\log C > c_1 \vee ( \log 4 + c_12^{-k_0}).
\end{equation}
 Let $k_{*} = k_{*}(m)$ be the smallest integer so that
\[
C^{2^{k_{*}+1}} \geq m^{1/2}.
\]
Note that this makes
\(
k_* = \frac{\log\log m}{\log 2} + \Theta(1).
\)

\begin{lemma}
\label{lem:diagonal_recurrence}
With high probability,
for all $k_0 \leq k \leq k_{*}$ and for all $j$ with $m \geq j \geq \phi(m,k),$
\[
\frac{F_j(k)}{2j} \leq f(k).
\]
\end{lemma}
\begin{proof}
The case $k=k_0$ follows from Lemma~\ref{lem:proc_start} with $\omega(m) = \phi(m,k_0).$
%Base of induction: for $k = k_{0}$ statement follows from Lemma 1.2 applied to $\omega(m) = \phi(m,k_{0})$ and $\epsilon > 0$, such that $\alpha_{k_0} + \epsilon < \frac{1}{k_0^3}$.
%Step of induction: let statement correct for $k$. Prove it for $k+1$.
We now show how the proof follows by layered induction.  Let $\G_k$ be the event
\[
\G(k) = \{
F_j(k) \leq 2jf(k)\;,\;\forall\; j\; :\; m \geq j \geq \phi(m,k)
\}.
\]
For any $j \geq \phi(m,k+1),$
\[
\frac{F_{j}(k+1)}{2j} = 
\frac{F_{\phi(m,k)}(k+1)}{2j} 
+ \frac{1}{2j} \sum_{i=\phi(m,k)}^{j}\xi_{i}(k+1),
\]
%where $\xi_{i}(k+1)$ takes values $1$ and $k+1$ with probabilities $(F_{i}(k+1)/(2i))^{2}$ and $(F_{i}(k)/(2i))^{2}-(F_{i}(k+1)/(2i))^{2}$ and $0$ otherwise. 
where $\xi_{i}(k+1)=F_{i+1}(k+1) - F_{i}(k+1)$ follows the rule in~\eqref{eq:evolution}. 
Let $X_{j,k}$ be distributed as
\[
X_{j,k} \sim (k+1)\Binom(j-\phi(m,k),f(k)^2).
\]
Conditional on $\G(k),$ the sum $\sum_{i=\phi(m,k)}^{j}\xi_{i}(k+1)$ is stochastically dominated by $X_{j,k}.$

Consider the event 
\[
\mathcal{E}(k+1)=\{\exists j\geq \phi(m,k+1): \sum_{i=\phi(m,k)}^{j}\xi_{i}(k+1) > 3/2\Exp X_{j,k}\}.
\]
%$$\mathbb{P}(\exists l<m\geq 0:\sum_{i=\phi(m,k)}^{\phi(m,k+1)+l}\xi_{i}^{\prime}(k+1)>3/2(\phi(m,k+1)+l-\phi(m,k))f^{2}(k))\leq$$
%$$\sum_{l=0}^{m}e^{-(\phi(m,k+1)-\phi(m,k))f^{2}(k)-lf^{2}(k)} = e^{-(\phi(m,k+1)-\phi(m,k))f^{2}(k)}\sum_{l=0}^{m}e^{-lf^{2}(k)}\leq$$
%$$\frac{1}{1-e^{-f^{2}(k)}}e^{-(\phi(m,k+1)-\phi(m,k))f^{2}(k)}\leq$$
%$$\frac{e^{(-\log\log m)^{1/2}(C^{2^{k+2}}4^{k+1}-C^{2^{k+1}}4^{k})\exp( -c_1 2^{k-k_{0}})}}{\exp( -c_1 2^{k-k_{0}} )} \leq \frac{e^{(-\log\log m)^{1/2}(Ce^{-c_1})^{2^{k}}}}{\exp( -c_1 2^{k-k_{0}} )}.$$
On complement of $\mathcal{E}(k+1)$ we obtain, setting $l = j - \phi(m,k+1),$
\begin{align*}
\frac{F_{j}(k+1)}{2j}
&=\frac{F_{\phi(m,k)}(k+1)}{2\phi(m,k+1)+l} + \frac{1}{2(\phi(m,k+1)+l)} \sum_{i=\phi(m,k)}^{\phi(m,k+1)+l}\xi_{i}(k+1) \\
&\leq \frac{\phi(m,k)}{\phi(m,k+1)+l}+ \frac{3(k+1)(\phi(m,k+1)+l-\phi(m,k))f^{2}(k)}{4(\phi(m,k+1)+l)} \\
&\leq \frac{\phi(m,k)}{\phi(m,k+1)}+\frac{3}{4}(k+1)f^{2}(k)\\ 
&\leq C^{2^{k+1}-2^{k+2}}+\frac{3}{4}f(k+1)\\
&\leq \frac{1}{4}e^{-c_{1} 2^{k+1-k_{0}}}+\frac{3}{4}f(k+1) \\
&\leq f(k+1),
\end{align*}
where we have applied~\eqref{eq:Cdef} in the fifth line.
Hence we obtain that $\mathcal{E}(k+1)^c \subseteq \G(k+1),$ and thus we may bound
\begin{align*}
\Pr\left[\; \exists~k~, k_{*} \geq k > k_0~:~ \G(k)~\text{fails}
\right]
&\leq
\sum_{k=k_0}^{k_*-1} \Pr\left[
\G(k+1)^c \cap \G(k)
\right] \\
&\leq
\sum_{k=k_0}^{k_*-1} \Pr\left[
\mathcal{E}(k+1) ~\vert~ \G(k)
\right]. 
\end{align*}
We estimate the probability of this event conditional on $\G(k)$ using standard Chernoff bounds.  In the following $c>0$ is an absolute constant.
\begin{align*}
\Pr \left[ \mathcal{E}(k+1) \vert \G_k \right] 
&\leq \Pr\left[~\exists~j \geq \phi(m,k+1)~:~X_{j,k} > \tfrac32\Exp X_{j,k} \right] \\
&\leq \sum_{l=0}^{m}e^{-c(\phi(m,k+1)-\phi(m,k))f^{2}(k)-clf^{2}(k)} \\
&= e^{-c(\phi(m,k+1)-\phi(m,k))f^{2}(k)}\sum_{l=0}^{m}e^{-clf^{2}(k)} \\
&\leq \frac{1}{1-e^{-cf^{2}(k)}}e^{-c(\phi(m,k+1)-\phi(m,k))f^{2}(k)} \\
\intertext{ Here we use that $f(k) \leq f(k_0)$ and hence there is an absolute constant $C'$ so that $C'(1-e^{-cf^{2}(k)}) \geq f^{2}(k).$  Applying Lemma~\ref{lem:f_estimate},}
\Pr \left[ \mathcal{E}(k+1) \vert \G_k \right] 
&\leq C'\frac{e^{-\rho(m)(C^{2^{k+2}}-C^{2^{k+1}})\exp( -c_1 2^{k-k_{0}})}}{\exp( -c_1 2^{k-k_{0}} )} \\
&\leq C'\exp({-\rho(m)(Ce^{-c_1})^{2^{k}}+c_1 2^{k-k_{0}}}).
\end{align*}
%$$\frac{F_{j}(k+1)}{2j}=\frac{F_{\phi(m,k)}(k+1)}{2\phi(m,k+1)+l} + \frac{1}{2(\phi(m,k+1)+l)} \sum_{i=\phi(m,k)}^{\phi(m,k+1)+l}\xi_{i}(k+1)\leq$$
%$$\frac{F_{\phi(m,k)}(k)}{2\phi(m,k+1)+l} + \frac{k+1}{2(\phi(m,k+1)+l)} \sum_{i=\phi(m,k)}^{\phi(m,k+1)+l}\xi_{i}^{\prime}(k+1)\leq$$
%$$\frac{\phi(m,k)}{\phi(m,k+1)+l}+ \frac{3(k_{0}+1)(\phi(m,k+1)+l-\phi(m,k))f^{2}(k)}{4\phi(m,k+1)+l}\leq$$
%$$\frac{\phi(m,k)}{\phi(m,k+1)}+3/4(k+1)f^{2}(k)=1/4C^{-2^{k+1}}+3/4f(k+1) \leq$$
%$$1/4(e^{-c_{1}})^{2^{k+1}}+3/4f(k+1) = 1/4e^{-c_{1} 2^{k+1}}+3/4f(k+1) \leq$$ $$1/4e^{-c_{1} 2^{k+1-k_{0}}}+3/4f(k+1)\leq 1/4f(k+1)+3/4f(k+1)=f(k+1).$$
%Hence we obtain that if $k$ steps of induction are correct, then $k+1$ step is correct on complement of event $\mathcal{E}(k)$. 
Therefore we may conclude that
\begin{multline*}
\Pr\left[\; \exists~k~, k_{*} \geq k > k_0~:~ \G(k)~\text{fails}
\right]
\leq
\sum_{k=k_{0}}^{k_{*}}C'\exp({-\rho(m)(Ce^{-c_1})^{2^{k}}+c_1 2^{k-k_{0}}}).
\end{multline*}
It can be checked that for $m$ sufficiently large, this bound is monotone decreasing in $k,$ and hence we have that 
\begin{equation*}
\Pr\left[\; \exists~k~, k_{*} \geq k > k_0~:~ \G(k)~\text{fails}
\right]
\leq
k_{*}\exp({-A\rho(m)})
\end{equation*}
for some absolute constant $A.$  As $k_* = O(\log \log m),$ this tends to $0$ with $m,$ which completes the proof of Lemma~\ref{lem:diagonal_recurrence}. 

\end{proof}

As a consequence, we have that for $m \geq j \geq \phi(m,k)$, and all $k_{0}\leq k\leq k_{*}$
\[
F_j(k_{*}) \leq F_j(k) 
\leq 2j f(k)
\leq 2j \exp(-c_2 2^{k-k_0})
\]
with high probability. We can therefore find some constant $\beta>0$ so that 
\[
\frac{F_j(k_{*})}{2j}\leq \left(
\frac{\rho(m)}{\phi(m,k+1)}
\right)^{\beta} 
\]
for $m \geq j \geq \phi(m,k)$, and all $k_{0}\leq k\leq k_{*}$. For each $\log \log m  \leq j\leq m$ we could find $k$ such that $\phi(m,k+1) \geq j \geq \phi(m,k),$ which implies that there is a $\beta_0 > 0$ constant so that with high probability
\begin{equation}
\label{eq:decay}
\frac{F_j(k_{*})}{2j} \leq \frac{1}{j^{\beta_0}}
\end{equation}
for all $\log \log m / \log 2 \leq j \leq m.$  

A large enough value of $\beta_0$ would complete the proof.  If $\beta_0 > \tfrac12,$ then with high probability, $F_j(k_*+1)$ would be identically $0$ for all $j$ with high probability.  However, by the construction so far, it turns out $\beta_0$ must be strictly less than $\tfrac12.$  That said, it is possible to supercharge this result by letting the recurrence run a little farther.
\begin{lemma}
\label{lem:supercharger}
If there is an absolute constant $M_1 > 0$ so that with high probability for some $k \leq \log\log m,$
\[
F_j(k) \leq 2 j^{1-\beta},~~\forall~j:~m \geq j \geq (\frm)^{M_1},
\]
for some $\beta < \tfrac12,$
then there is an absolute constant $M_2 > 0$ so that with high probability
\[
F_j(k+1) \leq 2 j^{1-1.5\beta},~~\forall~j:~m \geq j \geq (\frm)^{M_2}.
\]
\end{lemma}
\begin{proof}
We let $\C$ be the event used as the hypothesis of the lemma.  Set $j_0 = (\frm)^{M_1}.$  Then for $j \geq j_0,$ we have that 
\[
\Delta_j =
F_j(k+1) - F_{j_0}(k+1)
\]
conditional on $\C$
is dominated by a sum of independent Bernoulli variables with means at most $j_0^{-2\beta}.$  Thus, we may find an absolute constant $c > 0$ so that $\Delta_j$ is stochastically dominated by Poisson variable $X_j$ with mean
\[
\Exp X_j = c\sum_{i=j_0}^j i^{-2\beta}
\leq \frac{c}{1-2\beta} j^{1-2\beta},
\]
with the inequality following by comparison with a Riemann sum.  From standard tail bounds for Poisson variables, we may find a constant $C'$ so that 
\[
\Pr \left[
\exists~j \geq j_0~:~ \Delta_j \geq C'\Exp X_j
\right] \leq 
C'\sum_{j=j_0}^\infty \exp(-j^{1-2\beta}/C'),
\] 
which is $o(1)$ using the hypothesis that $1-2\beta > 0.$
Thus it follows that with high probability
\begin{equation}
\notag
F_j(k+1)\leq F_{j_0}(k+1) + (k+1) C'j^{1-2\beta}.
\end{equation}
As $k \leq \log\log m,$ we have that $(k+1) \leq \frm$ for large enough $m.$  Choose $M_2$ sufficiently large that both of $j_0(\frm) \leq (\frm)^{M_2(1-1.5\beta)}$ and $C'(\frm) \leq (\frm)^{0.5M_2\beta}$ for all $m$ sufficiently large.  Then we conclude for all $j \geq (\frm)^{M_2},$
\begin{equation}
\notag
F_j(k+1)\leq j^{1-1.5\beta} + j^{0.5\beta}j^{1-2\beta} = 2j^{1-1.5\beta},
\end{equation}
as desired.

\end{proof}

\begin{lemma}
\label{lem:final_recurrence}
There is a $M=M(\beta_0) > 0$ and an integer $r=r(\beta_0) > 0$ so that setting 
$j_0 = (2\log\log m)^M,$ then with high probability
\[
F_j(k_{*} + r) = F_{j_0}(k_{*} + r)
\]
for all $m\geq j \geq j_0.$
\end{lemma}
\begin{proof}
We may apply Lemma~\ref{lem:supercharger} some $r'(\beta_0)$ many times to conclude that there is an $M=M(\beta_0)$ so that with high probability
\(
{F_j(k_{*}+r')} \leq 2j^{1-\beta}
\)
for all $j \geq j_0$ and for some $\beta > \tfrac 12.$   

Let $\C$ be the event
\[
\C = \{
{F_j(k_{*}+r')} \leq 2j^{1-\beta},~\forall~j \geq j_0
\}.
\]
It now follows from the usual recurrence argument that
\[
\Pr \left[
\exists~j;~j_{0} \leq j\leq m~:~
F_j(k_{*}+r'+1) > F_{j_0}(k_*+r'+1) \vert~\C
\right] =O\biggl(\sum_{i=j_0}^m i^{-2\beta}\biggr) = o(1),
\]
as $i^{-2\beta}$ is summable.  Thus taking $r=r'+1,$ we have shown the desired claim.
\end{proof}

We now prove the final theorem.
\begin{proof}[Proof of Theorem~\ref{thm:max_degree}]
From Lemma~\ref{lem:final_recurrence}, it follows that with high probability, \[
F_m(k_* + r) = F_{j_0}(k_* + r).
\]
As $F_{j_0}(k_* + r)$ is almost surely at most $O( (\log\log m)^{M+1}),$ it follows that the maximum degree of the $\pam$ graph after $m$ steps is $ (\log\log m)^{M+1}$ with high probability.  Note that $k_* + r = \Theta( j_0^{1/M_1} ),$ and hence with high probability, $\pam[j_0]$ has no vertices of degree $k_* + r.$  Thus in fact, it follows that with high probability $F_{j_0}(k_* + r) = 0,$ so that with high probability $F_{m}(k_* + r) = 0$ and the maximum degree of the graph is at most $k_* + r = \log\log m / \log 2 + \Theta(1).$  

%Suppose that $F_m(k)$ concentrates around some $f(k)m,$ in which case $F_{m+1}(k) - F_m(k)$ behaves like a series of i.i.d variables.   Then for consistency, we must have that 
%\[
%\Exp \left[ F_{m+1}(k) - F_m(k) \right] \approx f(k).
%\]
%As this is the case, we have that
%\[
%f(k) \approx \left( \frac{f(k)}{2} \right)^2 + \frac{f(k)( f(k-1) - f(k) )}{2}
%+ \left( \frac{f(k-1) - f(k)}{2} \right)^2
%.
%\]

We will now prove the lower bound. To do so we provide a coupling between the bin and ball model with two choices and our model. We will use Theorem 6 of~\cite{MRS01} for the lower bound estimate on the maximum degree. Let us recall the ball and bin model. Suppose that $n$ balls are sequentially placed into $n$ bins (denote them by $v_{1}$,...,$v_{n}$). Each ball is placed in the least full bin at the time of the placement, among 2 bins, chosen independently and uniformly at random. Theorem 6 of~\cite{MRS01} provides that in this case after all the balls are placed the number of balls in the fullest bin is at least $\log\log n/ \log 2 - \Theta(1)$ with high probability. With a slight change in the proof of this theorem it could be extended to $n$ bins and $\epsilon n$ balls with the same statement, where $0<\epsilon<1$ is some constant. From here we consider the model with $2m$ bins and $m$ balls and we will use extension of Theorem 6 of~\cite{MRS01} for $n=2m$ and $\epsilon=1/2$.

Let $N_{j}^{0}(k)$ be the number of bins that contain at least $k$ balls at time $j$. We will need the following lemma.
\begin{lemma}
\label{coupling_bab}
There is a coupling such that for all $k\geq 1$ and $1\leq j\leq m$ 
$$N_{j}^{0}(k)\leq F_{j}(k).$$
\end{lemma}
Note that with this lemma, the proof is now complete, as there is a $k'(m) = \log\log m / \log 2 - \Theta(1)$ so that with high probability $N_m^0(k') > 0.$  And so we turn to proving the lemma by induction over $j.$

When $j=1,$ the lemma is trivial, as 
\[
N_{1}^{0}(k)=\1\{k=1\}\leq 2\1\{k=1\}=F_{1}(k),\;k\geq 2.
\]
Suppose the statement is true for $j\leq j_{0}.$ We will show the construction can be extended to $j_{0}+1\leq m$. 
The difference $N_{j_{0}+1}^{0}(k)-N_{j_{0}}^{0}(k)$ takes value 1 with probability 
$$\frac{N_{j_{0}}(k-1)^{2}-N_{j_{0}}(k)^{2}}{n^{2}}\leq\frac{N_{j_{0}}(k-1)^{2}}{n^{2}}\leq \frac{F_{j_{0}}(k-1)^{2}}{n^{2}}.$$
If $j_{0}\leq n/2=m$ this probability does not exceed $\frac{F_{j_{0}}(k-1)^{2}}{(2j)^{2}}$, and hence the difference $N_{j_{0}+1}^{0}(k)-N_{j_{0}}^{0}(k)$ is stochastically dominated by $F_{j_{0}+1}(k)-F_{j_{0}}(k)$. Therefore there is a coupling such that $N_{j_{0}+1}^{0}(k)\leq F_{j_{0}+1}(k).$ 

%From extention of theorem 6 of [MRS01] it follows that there is constant $c$ such that with probability turn to $1$ as $m\rightarrow$ $N_{m}^{0}(\log \log m/\log 2-c)>0$. Thus using \ref{coupling_bab} we obtain than with high probability $F_{m}(\log \log m/\log 2-c)>0$, which mean than with high probability maximal degree of the graph exceeds $\log \log m/\log 2-c$.
%Proof of theorem 1 is complete.
\end{proof}

\section*{Acknowledgements.} 
The authors are grateful to Professor Itai Benjamini for suggesting the problem and for helpful discussions.
%\begin{appendices}
%\section{Appendix}
%\end{appendices}
%\section{References}
%[B95] Pobability and Measure, P. Billingsley, third edition, 1995.
%[FFF04] High degree vertices and eigenvalues in the preferential attachment graph, A. Flaxman, A. Frieze, T. Fenner, 2004.
%[HH13] Diameters in preferential attachment models, S. Hofstad, G. Hooghiemstra, 2013.
%[MRS01] The Power of Two Random Choices: A Survey of Techniques and Results, M. Mitzenmacher, A. Richa, R. Sitaraman, 2001.

\bibliographystyle{alpha}
\bibliography{po2c_pa}

\end{document}